\documentclass[12pt]{article}
\usepackage{amsmath,amssymb,amsthm,latexsym}
\usepackage[dvips]{graphicx}
\usepackage{color}
\usepackage{paralist}

\newcommand{\Fq}{{\bf F}_{q}\/}
\newcommand{\PG}{\mathrm{PG}}
\newcommand{\p}{{\bf P}}
\newcommand{\cK}{{\cal K}}
\newcommand{\cS}{{\cal S}}
\newcommand{\cP}{{\cal P}}
\newcommand{\cB}{{\cal B}}
\newcommand{\bsl}{\backslash}
\newcommand{\sqq}{{\sqrt{q}}}
\newcommand{\qa}{\textstyle\frac{1}{4}}
\newcommand{\txt}{\textstyle}
\newcommand{\pgg}{\PG(n,q)}
\newcommand{\ha}{{\textstyle\frac{1}{2}}}
\newcommand{\tqa}{\textstyle\frac{3}{4}}
\newcommand{\sha}{{\textstyle\frac{7}{2}}}

\newcommand{\elha}{{\textstyle\frac{11}{2}}}

\newcommand{\cO}{{\cal O}}

\newcommand{\gs}{\sigma}
\newcommand{\gt}{\tau}

\newcommand{\Fqn}{{\bf F}_{q^n}}
\newcommand{\bF}{{\bf F}}
\newcommand{\F}{${\bf F}_{q}$\/}

\newcommand{\gD}{\Delta}
\newcommand{\gd}{\delta}
\newcommand{\gz}{\zeta}
\newcommand{\cH}{{\cal H}}
\newcommand{\cE}{{\cal E}}
\newcommand{\cU}{{\cal U}}
\newcommand{\gO}{\Omega}

\newcommand{\wtl}{\widetilde}
\newcommand{\g}{\gamma}
\newcommand{\cV}{{\cal V}}

\newcommand{\ov}{\overline}
\newcommand{\cW}{{\cal W}}
\newcommand{\cM}{{\cal M}}
\newcommand{\gth}{\theta}
\newcommand{\cC}{{\cal C}}

\definecolor{dmag}{rgb}{.5,0,.5}

\theoremstyle{plain}
\newtheorem{theorem}{Theorem}[section]

\newtheorem{cor}[theorem]{Corollary}

{\theoremstyle{definition}
\newtheorem{defn}[theorem]{Definition}
\newtheorem{example}[theorem]{Example}

\newtheorem{rem}[theorem]{Remark}

}

\title{Arcs, Caps and Generalisations in a Finite Projective Space}

\author{J.W.P. Hirschfeld and J.A. Thas}

\date{}

 \makeatletter
 \def\section{\@startsection {section}{1}{\z@}{-1.5ex plus -.5ex
 minus -.2ex}{1ex plus .2ex}{\large\bf}}

\begin{document}

\maketitle

\begin{tabular}{ll}
Department of Mathematics &\quad Department of Mathematics\\
University of Sussex&\quad Ghent University\\
Brighton  BN1 9QH &\quad 9000 Gent\\
United Kingdom &\quad Belgium
\end{tabular}


\vspace*{1mm}

\vspace*{3mm}


\begin{abstract}
Arcs and caps are fundamental structures in finite projective spaces. They can
be generalised. Here, a survey is given of some important results on these
objects, in particular on generalised ovals and generalised ovoids. The paper
also contains recent results  and several open problems. 
\end{abstract}

\section{Notation}
\begin{tabular}{ll}
$\Fq$  & the finite field of order $q$\\
$\PG(n,q)$ & the projective space of $n$ dimensions over $\Fq$\\
$\p(x_0,x_1,\ldots,x_n)$ & the point of $\PG(n,q)$ with coordinate vector 
$(x_0,x_1,\ldots,x_n)$\\
$\Pi_r$ & a subspace of dimension $r$ in $\PG(n,q)$\\
$V(k,q)$ & the vector space of $k$ dimensions over $\Fq$
\end{tabular}

\section{Introduction}
\label{intro}
A non-singular conic of the projective plane $\PG(2, q)$ over the Galois field
$\Fq$ consists of $q + 1$ points no three of which are collinear. It is
natural to ask if this non-collinearity condition for $q + 1$ points is
sufficient for them to be a conic. In other words, does this combinatorial
property characterise non-singular conics? For $q$ odd, this question was
affirmatively answered in 1954 by Segre \cite{segB54,segB55a}. 

Generalising, Segre considered sets of $k$ points in the $n$-dimensional
projective space $\PG(n, q)$, $k \geq 3$ and $n \geq 2$, no three of which are
collinear. For $n = 2$, such sets are \emph{$k$-arcs} of $\PG(2, q)$; for $n
\geq 3$ these sets are \emph{$k$-caps} of $\PG(n, q)$. Further, Segre
considers sets of $k$ points in $\PG(n,q),\ k \geq n+1,$ no $n+1$ of which lie
in a hyperplane; these sets are  \emph{$k$-arcs} of $\PG(n, q)$. There is a 
strong relation between arcs and algebraic curves, algebraic hypersurfaces 
and linear maximum distance separable (MDS) codes

Arcs and caps can be generalised by replacing their points with
$r$-dimensional subspaces to obtain \emph{generalised $k$-arcs} and
\emph{generalised $k$-caps} of $\PG(m, q)$ \cite{thaJ71a}. These have strong
connections to generalised quadrangles, projective planes, circle geometries,
strongly regular graphs, finite groups, and linear projective two-weight
codes. In this survey, results and problems concerning these objects are
discussed. The focus is on generalised ovals and generalised ovoids.

There is an enormous amount of material and results on arcs and caps in finite
projective spaces. On these, just a few important definitions and results are
mentioned. The emphasis is on some bounds, in particular on general bounds
that hold for all values of the size $q$ of the field, up to parity of $q$ and
a few exceptional small values. Some of these bounds, for example Theorems
5.25 and 5.26, cannot be found in earlier surveys or books. Many interesting
and strong reults on arcs are contained in 
\cite{balS12}, \cite{balS18}, \cite{balS12a}.

The main part of this paper is on generalised ovals and generalised ovoids.
Many characterisations and classifications are given. The last part is
focussed on the relations between certain generalised ovoids and finite
translation generalised quadrangles of order $(s,s^2)$. Also, the relationship
beween Moufang generalised quadrangles, generalised ovoids and the theorem of
Fong and Seitz on groups with a BN-pair of rank 2 is explained; see the recent
paper \cite{thaJ22a}.

Finally, several open problems are stated.

\section{Arcs, ovals and hyperovals in $\PG(2,q)$}
\label{sec3}

\begin{defn}
\begin{enumerate}[(1)]

\item  A {\em $k$-arc} in $\PG(2,q)$  is a set of $k$ points, with
$k\geq  3,$ such that no three of its points lie on a line.

\item An arc $\cK$ is {\em complete} if it is not properly contained in a larger
arc.

\item If $\cK\cup\{P\}$ is an arc for a point $P$ that is not in $\cK$, then $P$
\emph{extends} $\cK$.
   
\end{enumerate}
\end{defn}

\begin{theorem}{\rm (\cite[Chapter 8]{hirJ98})}

Let $\cK$ be a $k$-arc of $\PG(2,q)$. Then
\begin{enumerate}[\rm(i)]
\item $k\leq q+2;$

\item for $q$ odd$,\ k\leq q+1;$

\item any non-singular conic is a $(q+1)$-arc$;$

\item for $q$ even$,$ a  $(q+1)$-arc extends to a $(q+2)$-arc.
\end{enumerate}
\end{theorem}

\begin{defn}
In $\PG(2,q),$
\begin{enumerate}[(1)]
\item a $(q+1)$-arc is an {\em oval}$;$

\item a $(q+2)$-arc$,\ q$ even$,$ is a {\em complete oval} or
{\em hyperoval}.

\end{enumerate}

\end{defn}

\begin{theorem}{\rm (Segre \cite{segB54},\cite{segB55a};
\cite[Chapter 8]{hirJ98})}
In $\PG(2,q),\ q$ odd$,$ every oval is a non-singular conic.
\end{theorem}

\begin{rem}
For $q$ even, a non-singular conic extends to a hyperoval $\cK$. For $q\geq 8$, let
$\cK = C\cup\{P\}$, with $C$ a non-singular conic. If $P'\in C$, then
$\cK\bsl\{P'\}$ is an oval that is not a conic; this follows from the fact
that two distinct non-singular conics have at most four points in common. 
Hence, for $q$ even and  $q\geq 8$, not every oval is a conic. Also,
for $q$ even and  $q > 8$, there are many hyperovals that do not contain a
conic; see \cite{hirJ79,hirJ98}. 
\end{rem}

\begin{theorem}{\rm (Segre \cite{segB67}; \cite{thaJ87a}; 
\cite[Chapter 8]{hirJ98})}
\label{thm2.4.4}
\begin{enumerate}[\rm(i)]
\item For $q$ even$,$ a $k$-arc with
\[
k > q - \sqq + 1
\]
extends to a hyperoval.

\item For $q$ odd$,$ a $k$-arc with
\[
k > q - \qa\sqq + {\txt\frac{25}{16}}
\]
extends to an oval.
\end{enumerate}
\end{theorem}

\noindent
{\bf Open problem 1.} Classify all ovals and hyperovals for $q$ even.

\section{Arcs in $\PG(n,q),\ n\geq 3$ }
\label{sec4}
\begin{defn}
\begin{enumerate}[(1)]

\item  A {\em $k$-arc} in $\PG(n,q)$  is a set of $k$ points, with
$k\geq n+1\geq 3,$ such that no $n+1$ of its points lie in a hyperplane.

\item An arc $\cK$ is {\em complete} if it is not properly contained in a larger
arc.

\item Let $m(n,q)$ be the maximum size of a $k$-arc in $\pgg$. 

\item A {\em normal rational curve} of $\PG(n,q),\ n\geq 2,$ is any set of
points in $\PG(n,q)$ that is projectively equivalent to 
\[
\{\p(t^n,t^{n-1},\ldots,t,1)\mid t\in \Fq\}\cup \{\p(1,0,\ldots,0)\}.
\]
\end{enumerate}
\end{defn}

\begin{defn}
\begin{enumerate}[(1)]
\item With $V(m,q)$ the vector space of $m$ dimensions over $\Fq$, a 
\emph{linear code} $C$ is a subspace of  $V(m,q)$.

\item $C$ is an $[m,k,d]$ or an $[m,k,d]_q$ code if it has dimension $k$ and
minimum distance $d$, where the distance between distinct vectors
$(x_1,x_2,\ldots,x_m)$ and $(y_1,y_2,\ldots,y_m)$ of $C$ is the number of indices $i$
for which $x_i\neq y_i$ and with $d$ the minimum of these distances; here,
$d\leq m -k + 1$.

\item $C$ is \emph{maximum distance separable} (MDS) if  $d= m - k + 1$.
\end{enumerate}
\end{defn}

\begin{theorem}
For $m\geq 3,$ an {\rm MDS} $[m,k,d]$ code $C$ is equivalent to a  $k$-arc of $\PG(m-1,q)$.
\end{theorem}

\begin{proof}
Let $C$ be an $m$-dimensional subspace of $V(k,q)$ and let $G$ be an $m\times
k$ generator matrix for $C$; that is, the rows of $G$ form a basis for $C$.
Then $C$ is MDS if and only if any $m$ columns of $G$ are linearly independent;
this property is preserved under multiplication of the columns by non-zero
scalars. So, consider the columns of $G$ as points $P_1,P_2,\ldots,P_k$ of
$\PG(m-1,q)$. It follows that $C$ is MDS if and only if 
$\{P_1,P_2,\ldots,P_k\}$ is a $k$-arc of $\PG(m-1,q)$.
\vspace*{-3mm}
\end{proof}

\begin{theorem}
{\rm (\cite{thaJ69a}; Kaneta and Maruta \cite{kanH89a})}
For $\PG(n,q),$
take $q$ odd  and $n \geq 3$.
\begin{enumerate}[\rm(i)]
\item If $\cK$ is a $k$-arc with 
$k>q- \qa\sqrt{q}+n- \qa,$
then $\cK$ lies on a unique normal rational curve.

\item If $q > (4n -5)^2,$ every $(q+1)$-arc is a normal rational curve.

\item If $q>(4n -9)^2,$ then $m(n,q) = q +1$.

\end {enumerate}
\end{theorem}

\begin{theorem}
{\rm(Bruen et al. \cite{bruA88c}; Blokhuis et al. \cite{bloA90};
Segre \cite{segB67}; Storme and Thas \cite{stoL93a})}
\begin{enumerate}[\rm(i)]
\item For $q$ even$,\ q \neq 2,\ n\geq 3,$ if $\cK$ is a $k$-arc in
$\PG(n,q)$ with
 \[
k >  q -\ha \sqq + n -\tqa,
\]
then $\cK$ lies on a unique $(q+1)$-arc of $\PG(n,q)$.

\item A $(q+1)$-arc in $\PG(n,q),\ q$ even and $n\geq 4,$ with
\[
q > (2n-\sha)^2,
\]
 is a  normal rational curve.

\item If $\cK$ is a $k$-arc in $\PG(n,q),\ q$ even and $n\geq 4,$ with
\[
q > (2n-\elha)^2,
\]
then $k\leq q+1.$
\end{enumerate}
\end{theorem}

\begin{rem}
There are close relationships between $k$-arcs, algebraic curves and algebraic hypersurfaces.
\end{rem}

\section{Caps and ovoids}
\label{sec5}
\begin{defn}
\begin{enumerate}[(1)]

 \item In $\PG(n,q),\ n \geq 3$, a set $\cK$ of $k$ points no three of which
 are collinear is a $k$-\emph{cap}.

\item  A $k$-cap is \emph{complete} if it is not contained in a $(k+1)$-cap.

 \item A line of $\PG(n,q)$ is a \emph{secant}, \emph{tangent} or  \emph{external line} as 
it meets $\cK$ in $2,1$ or $0$ points.
 
\end{enumerate}

\end{defn}

\subsection{Caps and ovoids in $\PG(3,q)$}

\begin{theorem}
{\rm(\cite[Chapter 16]{hirJ85}; Bose \cite{bosR47}; Qvist \cite{qviB52})}
\label{thm3.2.1}
\begin{enumerate}[\rm(i)]
\item 
For a $k$-cap in $\PG(3,q)$ with $q\neq 2,$
\[
k\leq q^2 + 1.
\]

\item A $k$-cap in $\PG(3,2)$ has the bound  $k\leq 8;$ an $8$-cap is the complement 
of a plane.

\item Each elliptic quadric of $\PG(3,q)$ is a $(q^2 + 1)$-cap.
\end{enumerate}

\end{theorem}

\begin{defn}
A $(q^2 + 1)$-cap of $\PG(3,q),\ q\neq 2,$ is an \emph{ovoid}$;$ for $q=2,$
an {\em ovoid} is a set of $5$ points no $4$ of which are coplanar.
\end{defn}

\begin{theorem}
\label{thm3.2.2}
In $\PG(3,2),$ a complete cap is either an  ovoid$,$ which is an  elliptic quadric$,$ or an
$8$-cap$,$ which is the complement of a plane.
\end{theorem}






\begin{theorem}{\rm(\cite[Chapter 16]{hirJ85}; Barlotti \cite{barA55}; Panella \cite{panG55})}
\label{thm3.2.4}

In $\PG(3,q),\ q$ odd$,$ an ovoid is an elliptic quadric.

\end{theorem}

\begin{theorem}{\rm (Brown  \cite{broM00b})}
\label{thm3.2.5}
In $\PG(3,q),\ q$ even$,$ an ovoid containing at least one conic section is 
an elliptic quadric.

\end{theorem}





 \begin{theorem}{\rm(Tits \cite{titJ62a})}
\label{thm3.2.8}
\begin{enumerate}[\rm(i)]
\item For $q = 2^{2e+1},\ e\geq 1,$ the space $\PG(3,q)$ has ovoids that
are not elliptic quadrics. These are the \emph{Tits ovoids}.

\item 
With $q = 2^{2e+1},$ the canonical form of a Tits ovoid $\cO$  is the following$:$
\[
\cO = \{\p(1,z,y,x)\mid z = xy + x^{\gs + 2} + y^\gs\}\cup \{\p(0,1,0,0)\},
\]
where $\gs$ is the automorphism $t\mapsto t^{2^{e+1}}$ of $\Fq$.


\end{enumerate}

\end{theorem}

\begin{rem}
\label{rem3.2.9}
\begin{enumerate}[\rm(1)]
\item For $q$ even$,$ the only ovoids known are the elliptic quadrics and the
Tits ovoids.

\item For $q=4$, an ovoid of
$\PG(3,4)$ is an elliptic quadric; see Barlotti \cite{barA55} or \cite{hirJ85}.

\item For $q=8$, an ovoid is an elliptic quadric or a Tits ovoid; see Segre
\cite{segB59b} and Fellegara \cite{felG62}.

\item For $q=16,32$, all ovoids were determined by O'Keefe, Penttila and Royle; see 
\cite{okeC90b}, \cite{okeC92c}, \cite{okeC94b}, \cite{hirJ15}. 

\item For $q=64$, an ovoid of $\PG(3,64)$ is an elliptic quadric; see
Penttila \cite{penT22}.
\end{enumerate}
\end{rem}

\begin{rem}
For the influence and many applications of the paper of Tits \cite{titJ62a}, see the recent paper 
\cite{thaJ23} by Thas and van Maldeghem.  

\end{rem}

\noindent
{\bf Open problem 2.} Determine all ovoids in $\PG(3,q)$ for $q$ even.

\subsection{Caps in $\PG(n,q),\ n\geq 3$}

\begin{defn}
$m_2(n,q)$ is the maximum size of a $k$-cap in $\PG(n,q)$.
\end{defn}

\begin{theorem}\label{t:6.5} {\rm (Hill \cite{hilR78a})}
\begin{enumerate}[\rm(i)]
\item  
$m_2(n, q) \leq qm_2(n - 1, q) - (q + 1), \mbox{for} \ n \geq 4, q > 2.$
\item 
$m_2(n, q) \leq q^{n - 4}m_2(4, q) - q^{n - 4} - 2\,{\frac{q^{n - 4} - 1}{q - 1}} + 1,\ \mbox{for}\ n \geq 5, q > 2.$
\end{enumerate}
\end{theorem}

\begin{rem}
These results were obtained using the theory of cap-codes.
\end{rem}

Exact values of $m_2(n, q)$ are known in just a few cases.

\begin{theorem}\label{t:7.1}
\begin{enumerate}[\rm(i)]
\item {\rm (Bose \cite{bosR47})} $m_2(n, 2) = 2^n;$ a $2^n$-cap of $\PG(n, 2)$ is the complement of a hyperplane.

\item  {\rm (Pellegrino \cite{pelG70})} $m_2(4, 3) = 20;$ there are nine projectively distinct $20$-caps in $\PG(4, 3)$.

\item  {\rm (Hill \cite{hilR73})} $m_2(5, 3) = 56;$ the $56$-cap in $\PG(5, 3)$ is projectively unique.

\item  {\rm (Edel and Bierbrauer \cite{edeY99})} $m_2(4, 4) = 41;$ there exist two projectively distinct $41$-caps in $\PG(4, 4)$.
\end{enumerate}
\end{theorem}

\begin{rem}\label{r:7.2}
No other values of $m_2(n, q), n > 3$, are known.
\end{rem}

 Several bounds were obtained for the number $k$ for which there exist complete $k$-caps in $\PG(3, q)$
which are not ovoids; these bounds are then used to determine bounds for $m_2(n, q)$, with $n > 3$. Here 
are a few good bounds, without restrictions on $q$ except for a few small cases.

\begin{theorem}\label{t:7.3} {\rm (Meshulam \cite{mesR95})}
For $n \geq 4, q = p^h$ and $p$ an odd prime,
\begin{equation}
m_2(n, q) \leq {\frac{nh + 1}{(nh)^2}}q^n + m_2(n - 1, q).
\end{equation}
\end{theorem}

\begin{theorem}\label{t:7.4} {\rm(\cite[Chapter 18]{hirJ85})}
In $\PG(3, q)$, $q$ odd and $q \geq 67,$ if $\cK$ is a complete $k$-cap which is not an elliptic 
quadric$,$ then
\[
k < q^2 - {\txt\frac{1}{4}}q^{\frac{3}{2}} + 2q.
\]
More precisely,
\[
k \leq {q^2 - {\txt\frac{1}{4}}q^{\frac{3}{2}} + R(q)},
\]
where
\[
R(q) = \frac{(31q + 14\sqrt{q} -53)}{16}.
\]
\end{theorem}

\begin{defn}\label{d:7.5}
Let $m_2^\prime(2, q)$ be the size of the second largest complete arc of
$\PG(2, q)$ and let $m_2^\prime(3, q)$ be the size of the second
largest cap of $\PG(3, q)$.
\end{defn}

\indent Nagy and Sz\H{o}nyi \cite{nagG97} 
follow more or less the line of the proof of Theorem \ref{t:7.4}, and derive a bound
for $m_2^\prime(3, q)$ in terms of $m_2^\prime(2, q)$. Their bound involves a
more careful enumeration of certain plane sections of a large cap; so it
yields an improvement on the bounds in Theorem \ref{t:7.4}.

\begin{theorem}\label{t:7.6} {\rm (Nagy and Sz\H{o}nyi \cite{nagG97})}
If$,$ for $q$ odd$,$ $m_2^\prime(2, q) \geq (5q + 19)/{6},$ then
\[
m_2^\prime(3, q) < qm_2^\prime(2, q) +  \tqa(q + {\txt\frac{10}{3}} - m_2^\prime(2, q))^2  - q - 1.
\]
\end{theorem}

\begin{theorem}\label{t:7} {\rm (\cite{thaJ18a})}
In $\PG(3, q),$ $q$ even and $q \geq 8,$ if $\cK$ is a complete $k$-cap which is not an ovoid$,$ then
\[
k < q^2 - (\sqrt{5} - 1)q + 5.
\]
\end{theorem}

\begin{rem}\label{r: 7.8}
Combining the previous theorem with the main theorem of Storme and Sz\H{o}nyi
\cite{stoL93}, there is an immediate improvement of the previous result. This
important remark is due to Sz\H{o}nyi.
\end{rem}

\begin{theorem}\label{t:7.9} {\rm(\cite{thaJ18a})}
In $\PG(3, q),\ q$ even and $q \geq 2048,$ if $\cK$ is a complete $k$-cap which
is not an ovoid$,$ then
\[
k < q^2 - 2q +3\sqrt{q} + 2.
\]
\end{theorem}

Relying on Theorems \ref{t:6.5} and \ref{t:7.4},  the following result is obtained.

\begin{theorem}\label{t:7.10} 
{\rm(\cite{hirJ83a,hirJ16})}
In $\PG(n, q),\ n \geq 4,$ $q \geq 197$ and odd$,$
\[
m_2(n, q) < q^{n - 1} - \qa q^{n - \frac{3}{2}} + 2q^{n - 2}.
\]
For $n \geq 4$, $q \geq 67$ and odd$,$
\begin{eqnarray*}
m_2(n, q) & < & q^{n - 1} - q^{n - \frac{3}{2}} + {\txt\frac{1}{16}}(31q^{n - 2}  + 22q^{n - \frac{5}{2}} - 112q^{n - 3} - 14q^{n - \frac{7}{2}} 
+ 69q^{n - 4})\\ 
 && - 2(q^{n - 5} + q^{n - 6} + \cdots + q + 1),\nonumber
\end{eqnarray*}
where there is no term   $-2(q^{n - 5} + q^{n - 6} + \cdots + q + 1)$ for $n = 4$.
\end{theorem}
 Relying on Nagy and Sz\H{o}nyi \cite{nagG97}, the following improvement of Theorem \ref{t:7.10} is obtained.

\begin{theorem} \label{t:7.11} 
{\rm (Storme,  Thas and Vereecke \cite{stoL02})} 
If$,$ for $q$ odd$,$
\[
m_2^\prime(2, q) \geq \frac{5q + 25}{6} \ \mbox{and} \ m_2(4, q) > \frac{41q^3 + 202q^2 - 47q}{48},
\]
then
\[
m_2(4, q) < (q + 1)(qm_2^\prime(2, q) + \tqa(q + {\txt\frac{10}{3}} - m_2^\prime(2, q))^2 - q - 1 - m_2^\prime(2, q)) + m_2^\prime(2, q).
\]
\end{theorem}

\indent Bounds for $m_2(n, q)$, $n > 4$ and $q$ odd, can now be calculated using Hill's
Theorem \ref{t:6.5}.

\begin{rem}\label{r:7.12}
In  \cite{stoL02},
small improvements of Theorem \ref{t:6.5} are obtained.
\end{rem}

Relying on Theorems \ref{t:6.5} and \ref{t:7.9}, the following results are obtained.

\begin{theorem}\label{t:7.13} 
{\rm(\cite{thaJ18b})} 
\begin{enumerate}[\rm(i)]
\item  $m_2(4, 8) \leq 479$.
\item  $m_2(4, q) < q^3 - q^2 + 2\sqrt{5}q - 8$, $q$ even, $q > 8$.
\item  $m_2(4, q) < q^3 - 2q^2 + 3q\sqrt{q} + 8q - 9\sqrt{q} - 2$, $q$ even, $q \geq 2048$.
\end{enumerate}
\end{theorem}

\begin{theorem}\label{t:7.14} 
{\rm(\cite{thaJ18b})} 
For $q$ even$,\ q > 2,\ n \geq 5,$
\begin{enumerate}[\rm(i)]
\item  $m_2(n, 4) \leq {\frac{118}{3}}4^{n - 4} + \frac{5}{3};$

\item  $m_2(n, 8) \leq 478.8^{n - 4} - 2(8^{n - 5} + 8^{n - 6} + \cdots + 8 + 1) + 1;$

\item  $m_2(n, q) < q^{n - 1} - q^{n - 2} + 2\sqrt{5}q^{n - 3} - 9q^{n - 4} - 2(q^{n - 5} + 
q^{n - 6} + \cdots + 1) + 1,$ 
for $q > 8;$

\item  
$
m_2(n, q)  < q^{n - 1} - 2q^{n - 2} + 3q^{n - 3} \sqrt{q} + 8q^{n - 3} - 9q^{n - 4}\sqrt{q} - 7q^{n - 4}$

$\hspace*{2cm} - 2(q^{n - 5} + q^{n - 6} + \cdots + q + 1) + 1,$
 
for $q \geq 2048$.
\end{enumerate}
\end{theorem}

\begin{rem}\label{r:7.15}
For a  survey  on caps, see Hirschfeld and Storme  
\cite{hirJ01c}.
\end{rem}

\section{Generalised ovals}
\label{sec6}

\subsection{Introduction}
Arcs, ovals and hyperovals   can be generalised by replacing 
their points with $(n-1)$-dim\-ensional subspaces, $n\geq 1$,  to get
\emph{generalised $k$-arcs}, \emph{pseudo-ovals} and \emph{pseudo-hyperovals}.

These objects were defined in 1971 by Thas \cite{thaJ71a}. In 1973, Thas
\cite{thaJ73g}, the relation between pseudo-ovals and generalised
quadrangles was discovered. In 1974, Thas \cite{thaJ74f} showed that these
pseudo-ovals play a key role in the theory of translation generalised
quadrangles with the same number of points and lines.

\subsection{Generalised $k$-arcs}

\begin{defn}

\begin{enumerate}[\rm(1)]
\item A \emph{generalised $k$-arc} in $\PG(3n - 1,q)$ is a set $\cK$ of
$(n-1)$-dimensional subspaces, with $|\cK| =k\geq 3$, such that no three of
its elements lie in a hyperplane.

\item $\cK$ is \emph{complete} if it is not properly contained in a larger
generalised arc.
\end{enumerate}

\end{defn}

\begin{example}
\begin{enumerate}[\rm(1)]

\item For $n=1$, then the $k$-arcs of $\PG(2,q)$ arise.

\item For $n=2$, then $\cK$ is a set of  $k$ lines in $\PG(5,q)$ such that
every three generate the space.
\end{enumerate}

\end{example}

\begin{theorem} {\rm(\cite{thaJ71a})}
\begin{enumerate}[\rm(i)]
\item For every generalised $k$-arc in $\PG(3n - 1,q),$
\begin{enumerate}[\rm(a)]
\item $k\leq q^n + 2;$

\item $k\leq q^n + 1$ when $q$ is odd.

\end{enumerate}

\item In $\PG(3n - 1,q),$ there exist $(q^n + 1)$-arcs for every $q;$ for
$q$ even$,$ there exist $(q^n + 2)$-arcs.

\item If $O$ is a generalised $(q^n + 1)$-arc in $\PG(3n - 1,q),$ then each element
$\pi_i$ of $O$ is contained in exactly one $(2n-1)$-dimensional subspace
$\gt_i$ that is disjoint from all elements of $O\bsl\{\pi_i\};$ here$,$ $\gt_i$ is
the \emph{tangent space} of $O$ at $\pi_i$. 

\item For $q$ even$,$ all tangent spaces of a generalised 
$(q^n + 1)$-arc $O$ of $\PG(3n - 1,q)$ contain a common 
$(n-1)$-dimensional space $\pi,$ the \emph{nucleus} of $O$. Hence $O$ is
not complete and extends to a generalised 
$(q^n + 2)$-arc by adding its nucleus. 
\end{enumerate}

\end{theorem}

\subsection{Pseudo-ovals and pseudo-hyperovals}

\begin{defn}
\begin{enumerate}[\rm(1)]
\item A generalised $(q^n + 1)$-arc of $\PG(3n - 1,q)$ is
a \emph{generalised oval} or \emph{pseudo-oval}  or
 $[n-1]$-\emph{oval}  of $\PG(3n - 1,q)$. For $n=1$, a 
pseudo-oval is just an oval  of $\PG(2,q)$.

\item With $q$ even, a generalised $(q^n + 2)$-arc of $\PG(3n - 1,q)$ is
a \emph{generalised hyperoval} or \emph{pseudo-hyperoval}  or
 $[n-1]$-\emph{hyperoval}  of $\PG(3n - 1,q)$. For $n=1,$ a 
pseudo-hyperoval is just a hyperoval  of $\PG(2,q)$.

\end{enumerate}

\end{defn}

\begin{theorem}{\rm (Payne and Thas \cite{payS84,payS09})}
\begin{enumerate}[\rm(i)]

\item In $\PG(3n - 1,q),$ each hyperplane not containing a tangent space of
the pseudo-oval $O$ contains either $0$ or $2$ elements of $O$. When $q$ is
even$,$ each hyperplane contains either $0$ or $2$ elements of a 
pseudo-hyperoval.

\item For $q$ odd$,$ each point of $\PG(3n - 1,q)$ not contained in an element
of the pseudo-oval $O$ belongs to either $0$ or $2$ tangent spaces of $O$.

\end{enumerate}

\end{theorem}

\subsection{Regular pseudo-ovals and regular pseudo-hyperovals}

In the extension $\PG(3n - 1,q^n)$ of the space $\PG(3n - 1,q)$, take $n$
planes $\xi_i,\ i=1,2,\ldots,n$, that are conjugate for the extension $\Fqn$ of
$\Fq$ and which span $\PG(3n - 1,q^n)$. Thus they form an orbit of
the Galois group corresponding to the extension and span $\PG(3n - 1,q^n)$.

In $\xi_1$, consider an oval 
$$
O_1 = \{P_0^{(1)},P_1^{(1)},\ldots,P_{q^n}^{(1)}\}
$$ 
or a hyperoval 
$$
O_{1}' = \{P_0^{(1)},P_1^{(1)},\ldots,P_{q^n +1}^{(1)}\}.
$$ 
Next, for $i = 0, 1, \ldots, q^n$ or $i = 0, 1, \ldots, q^n + 1$,
let $P_i^{(1)},P_i^{(2)},\ldots,P_i^{(n)}$ be conjugate in $\bF_{q^n}$ over
$\Fq$. These points define an $(n-1)$-dimensional subspace $\pi_i$ over $\Fq$. Consequently, 
$O = \{\pi_0, \pi_1,\ldots,\pi_{q^n} \}$ is a pseudo-oval and $O' = \{\pi_0, \pi_1,\ldots,\pi_{q^n +1} \}$ 
 is a  pseudo-hyperoval of $\PG(3n - 1,q^n)$.


These are the \emph{regular} or \emph{elementary} pseudo-ovals and the 
\emph{regular} or \emph{elementary} pseudo-hyperovals of
$\PG(3n - 1,q)$. If $O_1$ is a conic in $\PG(2,q^n)$, then the corresponding
 pseudo-oval is a \emph{classical pseudo-oval} or \emph{pseudo-conic}.
 
Alternatively, let $V$ be the vector space over $\Fqn$ underlying the projective
plane $\PG(2,q^n)$. If $V$ is considered as an \F-vector space, each point
of  $\PG(2,q^n)$ becomes an $(n-1)$-dimensional subspace of $\PG(3n - 1,q)$.
If $O_1$ is an oval or hyperoval of $\PG(2,q^n)$, then here it becomes a 
regular pseudo-oval or regular pseudo-hyperoval of $\PG(3n - 1,q)$.

\begin{rem}
Every known  pseudo-oval and pseudo-hyperoval is regular. By Segre's theorem,
for $q$ odd every regular  pseudo-oval is a pseudo-conic.  
\end{rem}

\noindent
{\bf Open problem 3.}
Is every pseudo-oval regular? Is every pseudo-hyperoval regular?

\begin{theorem}{\rm (Payne and Thas \cite{payS84,payS09})}
For $q$ odd, the tangent spaces of a pseudo-oval $O$ in $\PG(3n - 1,q)$ 
form  a pseudo-oval $O^*$ in the dual space of $\PG(3n - 1,q)$.
\end{theorem}

\begin{defn}
The pseudo-oval $O^*$ is the \emph{translation dual} of the pseudo-oval $O$.
\end{defn}

\noindent
{\bf Open problem 4.} For $q$ odd, is every pseudo-oval $O$ isomorphic to its
translation dual?

\section{Characterisations}
\label{sec7}

\subsection{Pseudo-ovals, pseudo-hyperovals and spreads}
Let $O = \{\pi_0, \pi_1,\ldots,\pi_{q^n}\}$ be a pseudo-oval in $\PG(3n - 1,q)$.
The tangent space of $O$ at $\pi_i$ is $\gt_i$. Choose $i \in \{0, 1, \ldots,
q^n\}$ and let $\Pi_{2n-1}$ be skew to $\pi_i$. Further, let 
$\gt_i\cap \Pi_{2n-1} = \eta_i$ and 
$\langle \pi_i,\pi_j \rangle \cap \Pi_{2n-1} =\eta_j$ for $j\neq i$; here,
$\langle \pi_i,\pi_j \rangle$ is the subspace generated by $\pi_i$ and $\pi_j$.
Then $\{\eta_0, \eta_1,\ldots,\eta_{q^n} \} = \gD_i$ is an $(n-1)$-spread of $\Pi_{2n-1}$; 
that is, the elements of  $\gD_i$ partition $\Pi_{2n-1}$.

Now, let $q$ be even and let $\pi$ be the nucleus of $O$. Let $\Pi_{2n-1} \subset \PG(3n - 1,q)$ 
be skew to $\pi$.
If $\xi_i = \Pi_{2n-1}\cap \langle\pi,\pi_i \rangle$, then 
$\{\xi_0, \xi_1,\ldots,\xi_{q^n} \} = \gD$ is an $(n-1)$-spread of $\Pi_{2n-1}$.

Next, let $q$ be odd. Choose $\gt_i$ for $i \in \{0, 1, \ldots, q^n\}$. 
If $\gt_i \cap \gt_j = \gd_j$ with $j\neq i$, then
\[
\{\gd_0,\gd_1,\ldots,\gd_{i-1},\pi_i,\gd_{i+1},\ldots,\gd_{q^n}\} = \gD_i^*
\] 
is an $(n-1)$-spread of $\gt_i$.


\begin{defn}
Let $V$ be the $2$-dimensional vector space defining the projective line
$\PG(1,q^n)$. Considering $V$ as an $\Fq$-vector space, each point of
$\PG(1,q^n)$ becomes an $(n-1)$-dimensional subspace of $\PG(2n - 1,q)$. The
$(n-1)$-spread of $\PG(2n - 1,q)$ consisting of these $q^n + 1$ subspaces is a
\emph{regular spread} of $\PG(2n - 1,q)$. 
\end{defn}

\begin{theorem}{\rm (Casse, Thas and Wild \cite{casL85})} 
Let $O$ be a pseudo-oval of $ \PG(3n - 1,q),$ with $q$ odd. Then at least one of the 
$(n-1)$-spreads
\[
\gD_0,\gD_1,\ldots,\gD_{q^n},\gD_0^*,\gD_1^*,\ldots,\gD_{q^n}^*
\]
is regular if and only if $O$ is regular$,$ that is$,$ if and only if $O$ is a pseudo-conic.
\end{theorem}

\begin{theorem}
\label{rvdv}
{\rm (Rottey, Van de Voorde \cite{rotS15a,rotS15b})}
Let $O$ be a pseudo-oval in $ \PG(3n - 1,q),$ with $q= 2^h,\; h> 1, \; n $ prime. Then all the 
$(n-1)$-spreads 
$
\gD_0,\gD_1,\ldots,\gD_{q^n}
$
are regular if and only if $O$ is regular.
\end{theorem}

\noindent
{\bf Open problem 5.} From this theorem, the following questions arise.
\begin{enumerate}[I.]
\item What happens when $q=2$?

\item What happens when $n$ is not prime?

\item What happens when not all spreads $\gD_i$ are regular?

\item What happens when at least one of the spreads $\gD_i$ is regular?
\end{enumerate}

\begin{rem}
In \cite{thaJ19}, a shorter proof of Theorem \ref{rvdv} is given and a slightly stronger result
obtained. Metsch and Van de Voorde \cite{metK20} used the considerations in   \cite{thaJ19}
to prove that it is sufficient to assume that at least $q^n - q + 3$ spreads are regular.
\end{rem}

\begin{defn} 
In $ \PG(3n - 1,q),$ let $\pi_1,\pi_2,\pi_3$ be mutually skew $(n-1)$-dimensional subspaces. Also, let 
$\gt_i$ be a $(2n-1)$-dimensional subspace containing $\pi_i$ but skew to $\pi_j$ and $\pi_k$, and let $\gt_i\cap \gt_j 
= \eta_k$ with $\{i,j,k\} = \{1,2,3\}$. The subspace generated by $\eta_i$ and $\pi_i$ is $\gz_i$. If the 
$(2n-1)$-dimensional spaces $\gz_1, \gz_2, \gz_3$ have an $(n-1)$-dimensional subspace in common, then 
$\{\pi_1, \pi_2, \pi_3\}$ and $\{\gt_1, \gt_2, \gt_3\}$ are in \emph{perspective}.
 \end{defn}

\begin{theorem}{\rm(\cite{thaJ11})}
Consider a pseudo-oval  $O =
\{\pi_0,\pi_1,\ldots,\pi_{q^n}\}$
of $ \PG(3n - 1,q),\ q$ odd$,$ and let $\gt_i$ be the tangent space of $O$ at $\pi_i$ for each $i$. If$,$ for
any three distinct $i,j,k,$ the triples 
$\{\pi_i, \pi_j,\pi_k\}$ and $\{\gt_i, \gt_j,\gt_k\}$ are in perspective$,$ then $O$ is a pseudo-conic. The converse also holds.  
\end{theorem}

\begin{rem}
By Segre \cite{segB54, segB55a}, for $n=1$ and $q$ odd, the triples $\{\pi_i,\pi_j,\pi_k\}$
and $\{\gt_i,\gt_j,\gt_k\}$ are always in perspective, and so $O$ is a conic. Hence, for $q$
odd, every oval is a conic. To prove that the two triples are in perspective, Segre uses his famous
Lemma of Tangents; see Lemma 8.11 of \cite{hirJ98}. What happens for $n>1$?
\end{rem}

\subsection{The case $n=2$}

For $n=2$, a pseudo-oval $O$ consists of $q^2 + 1$ lines of $\PG(5,q)$, every
three of which generate the space.

\begin{theorem}{\rm (Shult and Thas \cite{shuE94})}
If the pseudo-oval $O$ is contained in a non-singular hyperbolic quadric
$\cH(5,q),$ with $q$ odd$,$ then $O$ is a pseudo-conic.
\end{theorem}

Let $O$ be a pseudo-oval contained in a non-singular elliptic quadric
$\cE(5,q)$ of $\PG(5,q)$ with $q$ odd. It can be shown that $O$ is equivalent
to a set of $q^2 + 1$ points on the non-singular Hermitian variety
$\cU(3,q^2)$ of $\PG(3,q^2)$, with the property that no three of them are in a
common tangent plane of the variety; see, for example, Payne and Thas \cite{payS84,payS09}
 or Shult \cite{shuE05}.

Any pseudo-conic $O$ of $\PG(5,q),\ q$ odd, is the intersection of a
non-singular hyperbolic quadric $\cH(5,q)$ and a non-singular elliptic quadric
$\cE(5,q)$.

Bamberg, Monzillo and Siciliano \cite{bamJ21} showed that a pseudo-oval on
$\cE(5,q),\ q$ odd, is a subset of a five-class association scheme, defined on
certain line sets of $\cE(5,q)$. Pseudo-ovals and pseudo-conics are analysed
in terms of these association schemes.

\begin{rem}
For $q$ even, a pseudo-oval $O$ of $\PG(5,q)$ is never contained in a
non-singular quadric, since all tangent spaces of $O$ contain a common line;
see Shult and Thas \cite{shuE94} and Thas \cite{thaJ11}.
\end{rem}

\noindent
{\bf Open problem 6.} Is each pseudo-oval on $\cE(5,q),\ q$ odd, a pseudo-conic?

\section{Generalised ovoids}
\label{sec8}
\subsection{Introduction}

Ovoids can be generalised by replacing their points with $(n-1)$-dimensional
spaces, $n\geq 1$, to obtain \emph{generalised ovoids}. This generalisation
was first considered in 1971 by Thas \cite{thaJ71a}. However, this
generalisation was too restricted and, in 1974 \cite{thaJ74a}, it was shown
that these generalised ovoids are always of a very particular kind. The
appropriate definition of a generalised ovoid appeared in \emph{Finite
Generalized Quadrangles} by Payne and Thas \cite{payS84, payS09}.

\subsection{Pseudo-ovoids}

In $\gO =\PG(4n-1,q)$, let $O$ be a set of $(n-1)$-dimensional subspaces
$\pi_i,\ i= 0,1,\ldots,q^{2n}$, such that
\begin{enumerate}[(a)]
\item every three generate  a $\Pi_{3n-1}$;

\item for every $i \in \{0,1,\ldots,q^{2n} \}$, there is a
$(3n-1)$-dimensional subspace $\gt_i$ that contains $\pi_i$ and is disjoint
from $\pi_j$ for $j\neq i$.
\end{enumerate}
The space $\gt_i$ is the \emph{tangent space} of $O$ at $\pi_i$;
 it is uniquely defined by $O$ and $\pi_i$. 
   
\begin{defn}
The set $O$ is a \emph{generalised ovoid} or a \emph{pseudo-ovoid} or an \emph{egg} or an \emph{$[n-1]$-ovoid} of   
$\PG(4n - 1,q)$.
\end{defn}

\begin{example}
\begin{enumerate}[(1)]
\item When $n=1$, the ovoids  of $\PG(3,q)$ arise; the tangent spaces are planes.

\item When $n=2$, a pseudo-ovoid of $\PG(7,q)$ contains $q^4 + 1$ lines; the
tangent spaces are $5$-dimensional.

\end{enumerate}

\end{example}

\begin{theorem}{\rm (Payne and Thas \cite{payS84, payS09})}
\begin{enumerate}[\rm(i)]

\item Each hyperplane of $\PG(4n - 1,q)$ not containing a tangent space of the pseudo-ovoid $O$ contains exactly 
$q^n + 1$ elements of $O$.

\item Each point which is not contained in an element of $O$ is contained in exactly $q^n + 1$ tangent spaces. 

\end{enumerate}

\end{theorem}

\begin{cor}
\begin{enumerate}[\rm(i)]

\item Let $\wtl{O}$ be the union of all elements of a pseudo-ovoid $O$ in the space $\PG(4n - 1,q)$ and let $\Pi$ be any hyperplane. Then
$|\wtl{O}\cap \Pi| \in \{\g_1,\g_2\},$ with
\[
\g_1  = \frac{(q^n - 1)(q^{2n-1} + 1)}{q-1},\quad \g_1 - \g_2 = q^{2n-1}.
\]

\item {\rm (Delsarte \cite{delP72})} 
Hence $\wtl{O}$ defines a  projective $2$-weight linear code and a strongly regular graph. 
\end{enumerate}
\end{cor}

\subsection{Regular pseudo-ovoids}
In the extension $\PG(4n - 1,q^n)$ of $\PG(4n - 1,q)$, consider $n$ solids
$\xi_i,\ i= 1,2,\ldots,n,$ that are conjugate in the extension $\Fqn$ of $\Fq$
and which span $\PG(4n - 1,q)$. This means that they form an orbit of the
Galois group corresponding to this extension and span $\PG(4n - 1,q^n)$.

In the space $\xi_1$, take an ovoid $ O_1 = \{P_0^{(1)},
P_1^{(1)},\ldots,P_{q^{2n}}^{(1)} \}$. Next, let $P_i^{(1)},
P_i^{(2)},\ldots,P_i^{(n)},$ $i = 0,1,\ldots,q^{2n}$, be conjugate in $\Fqn$
over \F. The points $P_i^{(1)}, P_i^{(2)},\ldots,P_i^{(n)}$ now define an
$(n-1)-$dimensional subspace $\pi_i$ over \F\, for each $i = 0,1,\ldots,q^{2n}$.
It follows that the ovoid $O = \{\pi_0,\pi_1,\ldots,\pi_{q^{2n}} \}$ is a pseudo-ovoid of $\PG(4n
-1,q)$.

These are the \emph{regular} or \emph{elementary} pseudo-ovoids. If $ O_1$ is
an elliptic quadric over $\Fqn$, the corresponding pseudo-ovoid is \emph{classical}
or a \emph{pseudo-quadric}.
  
Alternatively, let $V$ be the $4$-dimensional vector space underlying the 
projective space $\PG(3,q^n)$. Considering $V$ as an \F-vector space, each
point of   $\PG(3,q^n)$ becomes an $(n-1)$-dimensional subspace of
 $\PG(4n -1,q)$. If $O_1$ is an ovoid of $\PG(3,q^n)$, then  $O_1$ becomes
 a regular pseudo-ovoid of $\PG(4n -1,q)$. 

\begin{rem}
For $q$ even, every known pseudo-ovoid is regular. For $q$ odd, there are
pseudo-ovoids that are not regular. By the theorem of Barlotti and Panella,
for $q$ odd, every regular pseudo-ovoid is a pseudo-quadric.
\end{rem}

\noindent
{\bf Open problem 7.} Is every pseudo-ovoid of $\PG(4n-1,q),\ q$ even, regular?

\subsection{Translation duals}

\begin{theorem}{\rm (Payne and Thas \cite{payS84, payS09})}
The tangent spaces of a pseudo-ovoid $O$ of $\PG(4n-1,q)$ form a pseudo-ovoid $O^*$ in the dual space of $\PG(4n-1,q)$. 
\end{theorem}

\begin{defn}
The pseudo-ovoid $O^*$ is the \emph{translation dual} of the pseudo-ovoid $O$.
\end{defn}

\noindent
{\bf Open problem 8.} For $q$ even, is every pseudo-ovoid $O$ of $\PG(4n-1,q)$ isomorphic to its translation dual?

\begin{rem}
For $q$ odd, there are pseudo-ovoids which are not isomorphic to their translation dual.   
\end{rem}

\subsection{Characterisations}

\begin{theorem}{\rm (Payne and Thas \cite{payS84, payS09})}
The pseudo-ovoid $O$ of $\PG(4n-1,q)$ is regular if and only if one of the following holds$:$
\begin{enumerate}[\rm(i)]
 \item for any point $P$ not contained in an element of $O,$ the $q^n + 1$ tangent spaces containing $P$ have exactly 
$(q^n -1)/(q-1)$ points in common$;$

\item each $\Pi_{3n-1}$ that contains at least three elements of $O$ contains exactly $q^n + 1$ elements of $O$.
\end{enumerate}
\end{theorem}

\begin{theorem}{\rm (Brown and Lavrauw \cite{broM04})}
A pseudo-ovoid $O$ in $\PG(4n-1,q),\ q$ even$,$ contains a pseudo-conic if and only if it is a pseudo-quadric.
\end{theorem}

\noindent
{\bf Open problem 9.} Is a pseudo-ovoid of $\PG(4n-1,q),\ q$ even$,$ containing a regular pseudo-oval always regular?

\subsection{Good pseudo-ovoids}

\begin{defn}
\begin{enumerate}[\rm(i)]
\item The pseudo-ovoid $O$ in $\PG(4n-1,q)$ is \emph{good} at its element $\pi$ if any $\Pi_{3n-1}$ containing $\pi$ and 
at least two other elements of $O$ contains exactly $q^n + 1$ elements of $O$.

\item In this case, $\pi$ is a \emph{good} element of $O,$ and $O$ is also 
said to be \emph{good}. 
\end{enumerate}

\end{defn}
A regular pseudo-ovoid is good at each of its elements.
\vspace*{2mm}

\begin{rem}
Every known pseudo-ovoid or its translation dual is good.
\end{rem}

\noindent
{\bf Open problem 10.} Is every pseudo-ovoid or its translation dual good? 

\begin{theorem}{\rm (\cite{thaJ94})}
For $q$ even$,$ if the pseudo-ovoid $O$ is good at the element $\pi,$  then the translation dual  $O^*$ is good at 
the tangent space of  $O$ at $\pi$.
\end{theorem}

\begin{rem}
For $q$ odd$,$ this theorem is not true.
\end{rem}

\noindent
{\bf Open problem 11.} For $q$ even$,$ is every good pseudo-ovoid $O$ of $\PG(4n-1,q)$ regular? 

\begin{theorem}{\rm (\cite{thaJ06})}
Let $O$ be a pseudo-ovoid of $\PG(4n-1,q),\ q$ even$,$ that is good at $\pi\in O$. If the $q^{2n} + q^n$ 
pseudo-ovals on $O$ containing $\pi$ are regular$,$ then $O$ is regular.
\end{theorem}

Now, the case that $q$ is odd is considered.

\begin{theorem}{\rm (\cite{thaJ94} )}
Let  the pseudo-ovoid $O$ of $\PG(4n-1,q),\ q$ odd$,$ be good at $\pi$. Then $\pi$ is contained
in exactly $q^{2n} + q^n$ pseudo-conics lying  on $O$.
\end{theorem}

Veronese surfaces play a key role in the theory of pseudo-ovoids.

\begin{defn}
In $\PG(5,K)$, with $K$ any field, in a suitable reference system, a \emph{Veronese surface} $\cV^4_2$ consists of the points
\[
\p(x_0^2,x_1^2,x_2^2,x_1x_2,x_2x_0,x_0x_1)
\]
with $x_0,x_1,x_2\in K$ and not all zero.
\end{defn}

\begin{theorem}
\begin{enumerate}[\rm(i)]

\item When $K = \Fq$, the surface $\cV^4_2$ contains $q^{2} + q + 1$ points and conics.

\item Any two of the points are contained in just one of these conics. 
\end{enumerate}
\end{theorem}
The planes containing the conics are the \emph{conic planes} of $\cV^4_2$. For
more on Veronese surfaces, see Hirschfeld and Thas \cite[Chapter 25]{hirJ91b}, 
 \cite[Chapter 4]{hirJ16}.

The classification of good pseudo-ovoids in $\PG(4n-1,q),\ q$ odd, is now given.

\begin{theorem}{\rm (\cite{thaJ97},\cite{thaJ99} )}
Let $O = \{\pi,\pi_1,\pi_2,\ldots,\pi_{q^{2n}}\}$ be a pseudo-ovoid in the space
$\PG(4n-1,q),\ q$ odd$,$ that is good at $\pi$. Then there are  three separate cases.
\begin{enumerate}[\rm(a)]
\item There exists $\Pi_3$ in $\PG(4n-1,q^n)$ which has exactly one point in
common with the extension $\ov{\pi}$ of $\pi$ to $\Fqn$ and with the extensions
$\ov{\pi_i}$ of $\pi_i$ to $\Fqn$ for $i = 1, 2, \ldots,q^{2n}$. These $q^{2n} + 1$
points form an elliptic quadric in $\Pi_3$ and $O$ is a pseudo-quadric.

\item There exists $\Pi_4$ in $\PG(4n-1,q^n)$ that intersects the extension
$\ov{\pi}$ of $\pi$ to $\Fqn$ in a line $\ell$ and which has exactly one point $R_i$ in
common with each $\ov{\pi_i},\ i = 1, 2, \ldots,q^{2n}$. Further$,$ let $\cW =
\{R_i\mid i = 1, 2, \ldots,q^{2n}\}$ and let $\cM$ be the set of all common
points of $\ell$ and the conics that contain exactly $q^n$ points of $\cW$.
Then the set $\cM\cup\cW$ is the projection of a Veronese surface $\cV^4_2$
from a  point $P$ in a conic plane $\gth$ of $\cV^4_2$ onto a hyperplane
$\Pi_4$ of the $\PG(5,q^n)$ containing $\cV^4_2;$ the point $P$ is an exterior
point of the conic $\cV^4_2\cap\gth$.

Here, $O$ is a non-classical \emph{Kantor-Knuth pseudo-ovoid}.

\item There exists $\Pi_5$ in $\PG(4n-1,q^n)$ that intersects the extension
$\ov{\pi}$ of $\pi$ in a plane $\mu$ and which has exactly one point $R_i$ in
common with each $\ov{\pi_i},\ i = 1, 2, \ldots,q^{2n}$. Let $\cW = \{R_i\mid
i = 1, 2, \ldots,q^{2n}\}$ and let $\cC$ be the set of all common points of
$\mu$ and the conics that contain exactly $q^n$ points of $\cW$. In this case$,$
$\cW\cup \cC$ is a Veronese surface in $\Pi_5$.
\end{enumerate} 
\end{theorem}

\begin{rem}
\begin{enumerate}[(1)]

\item For more on Kantor-Knuth pseudo-ovoids, see 
\cite{thaJ21} and \cite{thaJ06}.

\item A Kantor-Knuth pseudo-ovoid is isomorphic to its translation dual.
Conversely, when $q$ is odd and the good pseudo-ovoid $O$ is isomorphic to its
translation dual, then $O$ is classical or of Kantor-Knuth type; see Sections
4.12.4 and 5.1.4 of \cite{thaJ06}.

\item Each known example of Class (c) has $q= 3^h$.

\item Good pseudo-ovoids play a key role in the theory of translation
generalised quadrangles. They also give rise to new results for particular
point-sets in classical polar spaces, as well as to the construction of new
projective planes, new flocks of quadratic cones in $\PG(3,q^n)$ and new 
semifields. For a detailed study of the relation between these objects, see
Lunardon \cite{lunG97,lunG22}.

\end{enumerate}
\end{rem}

\noindent
{\bf Open problem 12.}
If the pseudo-ovoid $O$ of $\PG(4n-1,q),\ q$ odd,  is good, but neither classical nor of Kantor-Knuth type, is $q$ necessarily a power of $3$? 
It is  more difficult to classify all pseudo-ovoids in $\PG(4n-1,q),\ q$ odd, in Case (c). 

\begin{theorem}{\rm (Blokhuis,  Lavrauw and Ball \cite{bloA03})}
Suppose that the good pseudo-ovoid $O$ of $\PG(4n-1,q),\ q$ odd$,$ satisfies the inequality
\[
q\geq 4n^2 - 8n + 2.
\]
Then $O$ is classical or of Kantor-Knuth type.
\end{theorem}

\noindent
{\bf Open problem 13.}
Improve the inequality in this theorem.
\vspace*{2mm}

The known examples of pseudo-ovoids for $q$ odd are the following, 
\cite{thaJ06}.

\begin{enumerate}[(i)]

\item Kantor-Knuth pseudo-ovoids, including the classical ones.

\item Ganley and Roman pseudo-ovoids, for $q=3^h,\ h\geq 1$. They are
translation dual to each other; the Roman ones are due to Payne \cite{payS89}.

\item The Penttila-Williams-Bader-Lunardon-Pinneri pseudo-ovoid and its
translation dual, for $q = 3^5$.

\end{enumerate}

\section{Pseudo-ovoids in $\PG(2n + m -1,q),\ m\neq n$ and $n \geq 1$}
\label{sec9}

\begin{defn}
In $\gO =\PG(2n + m -1,q),\ m\neq n$ and $n \geq 1$, define a set $O = O(n,m,q)$ of subspaces as follows: $O$ is a set of 
$(n-1)$-dimensional subspaces $\pi_i,\ i = 0, 1, \ldots,q^{m}$, such that
\begin{enumerate}[(i)]
\item every three generate a $\Pi_{3n -1}$;

\item for every $i\in \{0, 1, \ldots,q^{m}\}$, there is an $(m + n-1)$-dimensional subspace $\gt_i$ that contains $\pi_i$ and
is disjoint from $\pi_j$ for $j\neq i$.
\end{enumerate}
\end{defn}
The space $\gt_i$ is the \emph{tangent space} of $O$ at $\pi_i$, and is uniquely defined by $O$ and $\pi_i$.

\begin{defn}
The set $O$ is a \emph{generalised ovoid} or a \emph{pseudo-ovoid} or an \emph{egg} or an \emph{$[n-1]$-ovoid} of
$\PG(2n +2m -1,q)$.
\end{defn}

\begin{rem}
For $m= 2n$, the pseudo-ovoids of Section 11 are obtained.
\end{rem}

\begin{theorem}{\rm (Payne and Thas \cite{payS84,payS09})}
\label{thm23}
\begin{enumerate}[\rm(i)]
\item $n < m \leq 2n$ for any  $O(n,m,q)$.

\item $n(c+1) = mc,$ with $c\geq 1$ odd.

\item $m = 2n$ for $q$ even. 

\end{enumerate}
\end{theorem}

\begin{defn}
$O(n,n,q)$ is a {\em pseudo-oval} in $\PG(3n -1,q)$.
\end{defn}

\noindent
{\bf Open problem 14.}
Does there exist an egg $O(n,m,q)$ for $q$ odd and $m \neq 2n$?





\section{Generalised quadrangles and the sets $O(n,m,q)$}
\label{sec10}

In this section, the equivalence between translation
generalised quadrangles and the sets $O(n,m,q)$ is explained.  

\begin{defn}
A finite {\em generalised quadrangle} (GQ) is an incidence structure
$\cS = (\cP,\cB,\mathbf{I})$ in which $\cP$ and $\cB$ are disjoint 
non-empty sets of {\em points} and {\em lines} for which $\mathbf{I}$
is a symmetric point-line incidence relation satisfying the following 
axioms:
\begin{enumerate}[(i)]

\item each point is incident with $t+1$ lines, $t\geq 1$, and two distinct
points are incident with at most one line;

\item each line is incident with $s+1$ points, $s\geq 1$, and two distinct
lines are incident with at most one point;

\item if $P$ is a point and $\ell$ is a line not incident with $P$, then there 
is always a unique pair $(Q,m)\subseteq \cP \times \cB$ for which 
$P\,{\bf I}\,m\,{\bf I}\,Q\,{\bf I}\,\ell$.
\end{enumerate}

\end{defn}
The integers $s$ and $t$ are the parameters of $\cS$, which has {\em order}
$(s,t)$; if $s=t$, then $\cS$ has {\em order} $s$.

There is a point-line duality for generalised quadrangles. This means that, in
any definition or theorem, interchanging `point' and `line' as well as
the parameters $s$ and $t$ gives a valid result.

Given two, not necessarily distinct, points $P$ and $Q$ of the generalised
quadrangle $\cS$, write $P\sim Q$ and say that $P$ and $Q$ are {\em
collinear}, provided that there is some line $\ell$ for which $P\,{\bf
I}\,{\ell}\,{\bf I}\,Q$; also, $P\not\sim Q$ means that $P$ and $Q$ are
not collinear. Dually, for $\ell,m\in \cB$, write $\ell\sim m$ or
$\ell\not\sim m$ as $\ell$ and $m$ are concurrent or not.

For $P\in \cP$, put $P^{\perp} = \{Q\in \cP\mid Q\sim P \}$; note that 
$P\in \cP^{\perp}$.
 
For terminology, notation and results on GQs, see the monographs \cite{payS84,payS09} of 
Payne and Thas.

\begin{theorem}
\begin{enumerate}[\rm(i)]
\item Let $\cS = (\cP,\cB,\mathbf{I})$ be a GQ of order $(s,t)$. If $|\cP| =v$ and 
 $|\cB| =b,$ then $v= (s+1)(st+1)$ and  $b= (t+1)(st+1)$; see {\rm 1.2.1} 
 of Chapter 1 in {\rm \cite{payS84,payS09}}. Also, 
$s+t$ divides $ st(s+1)(t+1)$; see {\rm 1.2.2} of Chapter 1 in {\rm\cite{payS84,payS09}}.

\item If $s > 1$ and $t > 1,$ then $t \leq s^2$ and dually $s \leq t^2$; see 
{\rm 1.2.3} of Chapter 1 in {\rm\cite{payS84,payS09}}.
\end{enumerate}
\end{theorem}

\begin{defn}
Let $\cS = (\cP,\cB,\mathbf{I})$ be a GQ of order $(s,t)$. If $P\sim Q,\, P\neq Q,$ 
or if $P\not\sim Q$ and $|\{P,Q\}^{\perp\perp}| = t + 1,$ then the pair $\{P,Q\}$ is 
{\em regular}. The point $P$ is {\em regular} if $\{P,Q\}$ is 
 regular for all $Q\in \cP,\,Q\neq P$.
\end{defn}

\begin{theorem}
If $\cS$ contains a regular pair $\{P,Q\},$ then either $s=1$ or $s\geq t$; see 
{\rm 1.3.6} of Chapter 1 in {\rm\cite{payS84,payS09}}. 
\end{theorem}

\begin{defn}
\begin{enumerate}[\rm(i)]
\item Let $\cS = (\cP,\cB,\mathbf{I})$ be a GQ of order $(s,t),$ with $s\neq 1,t\neq 1$. A 
collineation $\gth$ in $\cS$ is an {\em elation} about the point $P$ if $\gth =I$ or if $\gth$
fixes all lines incident with $P$ and fixes no point of $\cP\bsl P^{\perp}$. If there is 
a group $\cH$ of elations about $P$ acting regularly on $\cP\bsl P^{\perp}$, then $\cS$ 
is an {\em elation generalised quadrangle} {\rm(EGQ)} with {\em elation group} $\cH$ 
and {\em base point} $P$.

\item If the translation group $\cH$ is abelian$,$ then the {\rm EGQ} is a
{\em translation generalised quadrangle} {\rm(TGQ)} with {\em base point} or
{\em translation point} $P$ and {\em translation group} $\cH$. In this case$,$
$\cH$ is the set of all elations about $P$; see \rm{8.6.4} of \rm{Chapter 8} in
\rm{\cite{payS84,payS09}}. For any {\rm TGQ}, each line incident with the
base point is regular; so $t\geq s$.
\end{enumerate}
\end{defn}
\noindent
For a detailed study of TGQs, see the monograph \cite{thaJ06} by 
Thas, Thas and Van Maldeghem. 

\begin{defn}
The {\em kernel} $\bf{K}$ of the  \rm{TGQ} $\cS$ is the field with multiplicative 
group isomorphic to the group of all collineations of $\cS$ fixing line-wise 
the translation point $P$ and any given point $Q\not\sim P$. From  
\rm{\cite{payS84,payS09}},  $|{\bf K}|\leq s$.
\end{defn}

\begin{theorem}
Let $\cS$ be a {\rm TGQ} of order $(s,t),\, s\neq 1,t\neq 1$, with kernel $\bf{K}$.
If $\Fq$ is a subfield of $\bf{K},$ then $\cS$ corresponds to a set $O(n,m,q)$ 
with $s=q^n,\,t=q^m$. Conversely, to each set $O(n,m,q)$ corresponds a 
{\rm TGQ} $\cS$ of order $(s,t),$ with $s=q^n,\,t=q^m$, where $\Fq$ is a subfield 
of the kernel.
\end{theorem}

\begin{cor}
\begin{enumerate}[\rm(i)]
\item The theory of finite {\rm TGQs} is equivalent to the theory of the 
sets $O(n,m,q)$.

\item If $\cS$ is a {\rm TGQ} of order $(s,t),$ with $s\neq 1,\,t\neq
1,\,s=q^n,\,t=q^m$ and $n\neq m$, then$,$ by Theorem {\rm\ref{thm23}}$,$ 
$n<m\leq 2n,\ n(c+1) =mc$ with $c\geq 1$ odd$,$  and $2n =m$ for $q$ even.
 
\end{enumerate}
\end{cor}

\section{Pseudo-ovoids, Moufang quadrangles and Fong-Seitz}
\label{sec11}

This section contains recent developments on the relation between TGQs,  
sets $O(n,m,q)$, and finite groups.

 Let $\cS = (\cP,\cB,\mathbf{I})$ be a GQ of order $(s,t)$. Let $P\in\cP$ be fixed and define the 
following condition $(M)_P$.

\begin{defn}  
\begin{enumerate}[\rm(i)]
\item $(M)_P$: For any two distinct lines $\ell,\ell'$ of $\cS$ incident with
$P,$ the group of collineations of $\cS$ fixing $\ell$ and $\ell'$ point-wise
and $P$ line-wise is transitive on the set of lines distinct from $\ell$ and
incident with a given point $Q$ on $\ell$, with $Q\neq P.$

\item $(M)$: $\cS$ satisfies $(M)$ if it satisfies $(M)_P$ for all $P\in\cP$.

\item $\cS$ is a {\em Moufang} GQ if and only if it satisfies $(M)$ and its dual $(M')$. 
\end{enumerate}
\end{defn}  

\begin{theorem} {\rm(Tits \cite{titJ76a})}
Every finite Moufang {\rm GQ} is classical or dual classical, and conversely.
\end{theorem}

The classical and dual classical GQs are those arising from a quadric, a Hermitian variety, or a symplectic polarity.

\begin{rem}
Tits observes that this result follows from the classification in \cite{fonP73,fonP74} of all finite groups 
with a BN-pair of rank 2 having a Weyl group $D_8$.
\end{rem}

\begin{rem}
In their monograph \cite{payS84,payS09}, Payne and Thas made an almost
successful attempt to prove the Moufang theorem of Tits, in the finite case,
in a geometric way. To complete a geometric proof of this result, it would
suffice to show geometrically that each GQ of order $(s,s^2), s\neq 1,$ for
which each point is a translation point, is necessarily the GQ arising from an
elliptic quadric in $\PG(5,s)$. As a corollary, a purely geometric proof of a
large part of the theorem of Fong and Seitz would follow, namely in the case
where the Weyl group is $D_8$. 
\end{rem}

In \cite{thaJ22a}, the following stronger result on GQs of order $(s,s^2)$ is obtained.

\begin{theorem} {\rm(\cite{thaJ22a})}
\label{thm38}
Let $\cS$ be a {\rm TGQ} of order $(s,s^2), s> 1,$ having a regular line not incident
with the translation point and having $\cO = O(n,2n,q)$ as corresponding pseudo-ovoid
in $\PG(4n-1,q),\,q\neq 2$. Then $\cO$ is good.
\end{theorem}

Then, relying on results of Brown, Lavrauw, Lunardon, Payne, and Thas, the next theorem 
is obtained.

\begin{theorem} {\rm(\cite{thaJ22a})}
\label{thm39}
Let $\cS$ be a {\rm TGQ} of order $(s,s^2), s> 1,$ with translation point $P$
and  kernel $\bf{K}\neq \bF_2$. If $\cS$ has a regular line not incident with $P,$ then the 
following hold$:$
\begin{enumerate}[\rm(i)]
\item if $q$ is odd$,$ then $\cS$ is the point-line dual of the translation dual of a 
semifield flock {\rm TGQ}$;$

\item if $q$ is odd$,$ then $\cS$ is classical.
\end{enumerate}
\end{theorem}

\begin{rem}{\rm(\cite{thaJ22a})}
Theorems \ref{thm38} and \ref{thm39} have many implications. For example, for 
${\bf K}\neq \bF_2,$ the `missing part' in \cite{payS84,payS09}, a purely geometric proof 
of the theorem of Tits and a large part of the theorem of Fong and Seitz.
\end{rem}

\noindent
{\bf Open problem 16.}
What happens in the case $\bf{K} = \bF_2$?

\section{Weak generalised ovoids}
\label{sec12}

\begin{defn}
A \emph{weak generalised ovoid} of $\PG(4n-1,q)$ is a set of $(n-1)$-dimensional subspaces, $q^{2n} + 1$ in number,
such that any three generate a $\Pi_{3n -1}$.
\end{defn}

\noindent
{\bf Open problem 15.}
Is every weak generalised ovoid a generalised ovoid?
\vspace*{2mm}

\noindent
For results on weak generalised ovoids, see Rottey and Van de Voorde \cite{rotS15a,rotS15b}.
\vspace*{3mm}

\section*{Acknowledgements}

The research was funded by Ghent University and the University of
Sussex. 

\noindent 
There is no conflict of interest in this paper. All
data is available.

\end{document}